\documentclass[12pt]{article}
\usepackage{graphicx}
\usepackage{amsmath,amsthm,amssymb,amsfonts,euscript,enumerate}
\newtheorem{thm}{Theorem}[section]

\newtheorem{lem}[thm]{Lemma}

\newcommand{\mcD}{{\mathcal{D}}}

\newcommand{\R}{{\mathbb{R}}}
\newcommand{\Z}{{\mathbb{Z}}}

\newcommand{\1}{\partial}
\newcommand{\2}{\overline}
\newcommand{\3}{\varepsilon}
\newcommand{\4}{\widetilde}

\def\ni{\noindent}

\begin{document}
\title{Another proof of Ricci flow on incomplete\\ 
surfaces with bounded above Gauss curvature} 
\author{Shu-Yu Hsu\\
Department of Mathematics\\
National Chung Cheng University\\
168 University Road, Min-Hsiung\\
Chia-Yi 621, Taiwan, R.O.C.\\
e-mail: syhsu@math.ccu.edu.tw}
\date{Nov 10, 2010}
\smallbreak \maketitle
\begin{abstract}
We give a simple proof of an extension of the existence results of Ricci 
flow of G.~Giesen and P.M.~Topping \cite{GiT1},\cite{GiT2}, on incomplete 
surfaces with bounded above Gauss curvature without using the difficult 
Shi's existence theorem of Ricci flow on complete non-compact surfaces 
and the pseudolocality theorem of G.~Perelman \cite{P1} on Ricci flow. 
We will also give a simple proof of a special case of the existence 
theorem of P.M.~Topping \cite{T} without using the existence theorem of 
W.X.~Shi \cite{S1}.
\end{abstract}

\vskip 0.2truein

Key words: Ricci flow, incomplete surfaces, negative Gauss curvature

AMS Mathematics Subject Classification: Primary 58J35, 53C99 Secondary 35K55
\vskip 0.2truein
\setcounter{equation}{0}

\setcounter{equation}{0}
\setcounter{thm}{0}

Recently there is a lot of study on the Ricci flow on manifold by R.~Hamilton
\cite{H1}, \cite{H2}, S.Y.~Hsu [Hs1--4], G.~Perelman \cite{P1}, \cite{P2}, 
W.X.~Shi \cite{S1}, \cite{S2}, L.F.~Wu \cite{W1}, \cite{W2}, and others because 
Ricci flow is a powerful tool in the study of geometric problems. We refer the 
readers to the book \cite{CLN} by B.~Chow, P.~Lu and L.~Ni, for the basics 
of Ricci flow and the papers \cite{P1}, \cite{P2}, of G.~Perelman and
the book \cite{Z} by Qi S.~Zhang for the most recent results on Ricci flow.
 
In 1982 R.~Hamilton \cite{H1} proved that if $M$ is a compact manifold and
$g_{ij}(x)$ is a metric of strictly positive Ricci curvature, then there
exists a unique metric $g$ that evolves by the Ricci flow
\begin{equation}\label{Ricci-eqn}
\frac{\1 }{\1 t}g_{ij}=-2R_{ij}
\end{equation}
on $M\times (0,T)$ for some $T>0$ with $g_{ij}(x,0)=g_{ij}(x)$
where $R_{ij}(\cdot,t)$ is the Ricci curvature of $g(\cdot,t)$.

Short time existence of solutions of the Ricci flow on complete
non-compact Riemannian manifold with bounded curvature was proved by W.X~Shi
\cite{S1}. Global existence and uniqueness of solutions of the Ricci flow 
on non-compact manifold $\R^2$ was obtained by S.Y.~Hsu in \cite{Hs1}. 
Existence and uniqueness of the Ricci flow on incomplete surfaces with
negative Gauss curvature was obtained by  G.~Giesen and P.M.~Topping
in \cite{GiT1}. In \cite{GiT1} G.~Giesen and P.M.~Topping proved the 
following theorem.

\begin{thm}\label{main-thm}(Theorem 1.1 of \cite{GiT1})
Suppose $M$ is a surface (i.e. a 2-dimensional manifold without boundary)
equipped with a smooth Riemannian metric $g_0$ whose Gauss curvature satisfies
$K[g_0]\le -\eta<0$, but which need not be complete. Then exists a unique
smooth Ricci flow $g(t)$ for $t\in [0,\infty)$ with the following properties:

\begin{enumerate}

\item[(i)]$g(0)=g_0$;

\item[(ii)] $g(t)$ is complete for all $t>0$;

\item[(iii)] the curvature of $g(t)$ is bounded above for any compact time 
interval within $[0,\infty)$;

\item[(iv)] the curvature of $g(t)$ is bounded below for any compact time 
interval within $(0,\infty)$.

\end{enumerate} 

Moreover this solution satisfies $K[g(t)]\le-\frac{\eta}{1+t\eta}$ for 
$t\ge 0$ and $-\frac{1}{2t}\le K[g(t)]$ for $t>0$.

\end{thm}

By abuse of notation we will write $K[u]=K[g]$ for the Gauss curvature of 
a metric of the form $g=e^{2u}\delta_{ij}$. As observed by G.~Giesen and 
P.M.~Topping \cite{GiT1} in order to prove Theorem \ref{main-thm} it suffices 
to assume that $M=\mathcal{D}$ is a unit disk in $\R^2$ and 
$g_0=e^{2u_0}\delta_{ij}$ is a conformal metric on $\mathcal{D}$. Then by 
scaling Theorem \ref{main-thm} is equivalent to the following two theorems.

\begin{thm}\label{thm-a1}(cf. Theorem 3.1 and Lemma 2.2 of \cite{GiT1})
Let $g_0=e^{2u_0}\delta_{ij}$ be a smooth conformal metric on the unit disk 
$\mathcal{D}$ with \begin{equation}\label{K-time0-upper-bd}
K[u_0]\le -1. 
\end{equation}
Then there exists a smooth solution $g(t)=e^{2u}\delta_{ij}$ of 
\eqref{Ricci-eqn} in $\mathcal{D}\times [0,\infty)$ with $g(0)=g_0$ 
such that $g(t)$ is complete for every $t>0$ with the Gauss curvature $K[u(t)]$
satisfying
\begin{equation}\label{K-lower-bd}
K[u(t)]\ge -\frac{1}{2t}\quad\forall t>0,
\end{equation}
\begin{equation}\label{K-upper-bd}
K[u(t)]\le-\frac{1}{2t+1}\quad\forall t\ge 0,
\end{equation}
\begin{equation}\label{u-lower-bd}
u(x,t)\ge\log\frac{2}{1-|x|^2}+\frac{1}{2}\log (2t)\quad\forall x\in\mcD,
t>0,
\end{equation}
\begin{equation}\label{u-upper-bd}
u(x,t)\le\log\frac{2}{1-|x|^2}+\frac{1}{2}\log (2t+1)\quad\forall x\in\mcD,
t\ge 0,
\end{equation}
and 
\begin{equation}\label{u-u0-ineqn}
u(x,t)\ge u_0(x)+\frac{1}{2}\log(2t+1)\quad\mbox{ in }\mcD\times [0,\infty). 
\end{equation}
Moreover $g(t)$ is maximal in the sense that if $\4{g}(t)$ for $t\in [0,\3]$ 
is another Ricci flow with $\4{g}(0)=g_0$ for some $\3>0$, then
\begin{equation}\label{tilde-g-and-g-compare}
\4{g}(t)\le g(t)\quad\forall 0\le t\le\3.
\end{equation}
\end{thm}

\begin{thm}\label{thm-a2}(cf. Theorem 4.1 of \cite{GiT1})
Let $e^{2u_0}\delta_{ij}$ be a smooth metric on the unit disk $\mathcal{D}$ 
satisfying the upper curvature bound \eqref{K-time0-upper-bd}. Let 
$e^{2u}\delta_{ij}$ be a solution of \eqref{Ricci-eqn} in 
$\mathcal{D}\times (0,\infty)$ with $u(\cdot,0)=u_0$ which satisfies 
\eqref{K-lower-bd} and \eqref{K-upper-bd}. Then $u$ is unique among 
solutions that satisfy  \eqref{K-lower-bd} and \eqref{K-upper-bd}.
\end{thm}

The proof of Theorem 3.1 and Lemma 2.2 of \cite{GiT1} 
uses the results of \cite{T}, the Schwartz Lemma of S.T.~Yau \cite{Y}, 
and the difficult existence theorem of W.X.~Shi \cite{S1} for Ricci 
flow on complete non-compact manifolds. In this paper we will give a 
simple proof of Theorem \ref{thm-a1} using 
the results of K.M.~Hui in \cite{Hu3} and \cite{Hu4}. We will assume 
that $M=\mathcal{D}\subset\R^2$ is a unit disk for the rest of the paper. 
Note that for a metric $g$ on a 2-dimensional manifold, $\mbox{Ric}[g]=K[g]g$. 
We will also give simple proofs of the following extension of the
existence results of G.~Giesen and P.M.~Topping \cite{GiT2} and a 
special case of Theorem 1.1 of \cite{T} without using the existence theorem of 
W.X.~Shi \cite{S1} and the pseudolocality Theorem of G.~Perelman \cite{P1} 
on Ricci flow.

\begin{thm}\label{thm-a4}(cf. Theorem 3.1 of \cite{GiT2} and Theorem 1.1 
of \cite{T})
Let $g_0=e^{2u_0}\delta_{ij}$ be a smooth (possible incomplete) Riemannian 
metric on $\mcD$. Then there exists a maximal instantaneous smooth complete
Ricci flow $g(t)=e^{2u}\delta_{ij}$ on $\mcD$ for all time $t\in[0,\infty)$ with 
$g(0)=g_0$ which satisfies \eqref{K-lower-bd} and \eqref{u-lower-bd}. 
Suppose in addition the Gauss curvature satisfies
\begin{equation}\label{K-time0-upper-bd2}
K[g_0]\le K_0
\end{equation}
for some constant $K_0\ge 0$. Then the following holds.
\begin{enumerate}
\item[(i)] If $K_0>0$, then
\begin{equation}\label{K-upper-bd2}
K[u(t)]\le\frac{1}{K_0^{-1}-2t}\quad\forall 0\le t<(2K_0)^{-1}
\end{equation} 
and 
\begin{equation}\label{u-lower-bd1}
u(x,t)\ge u_0(x)+\frac{1}{2}\log(1-2K_0t)\quad\forall 0\le t<(2K_0)^{-1}.
\end{equation} 
\item[(ii)] If $K_0=0$, then
\begin{equation}\label{K-upper-bd3}
K[u(t)]\le 0\quad\forall t\ge 0
\end{equation} 
and 
\begin{equation}\label{u-lower-bd3}
u(x,t)\ge u_0(x)\quad\forall t\ge 0.
\end{equation} 
\end{enumerate}
\end{thm}

\begin{thm}\label{thm-a5}(cf. \cite{DP1}, \cite{Hu2}, \cite{Hs1},
Theorem 1.1 of \cite{T} and Theorem 3.2 of \cite{GiT2})
Let $g_0=e^{2u_0}\delta_{ij}$ be a smoooth metric on $\R^2$ which need not 
be complete. Then there exists a maximal instantaneous smooth complete 
Ricci flow $g(t)=e^{2u}\delta_{ij}$ on $\R^2$ for
$t\in [0,T)$ with $g(0)=g_0$ which satisfies \eqref{K-lower-bd} and
for any $0<T_1<T$ and $r_0>1$ there exists a constant $C>0$ such that
\begin{equation}\label{u-lower-bd2}
u(x,t)\ge -C-\log (|x|\log |x|)+\frac{1}{2}\log (2t)\quad\forall
|x|\ge r_0,0\le t\le T_1
\end{equation} 
where 
\begin{equation}\label{T-value}
T=\left\{\begin{aligned}
&\frac{\mbox{Vol}_{g_0}(\R^2)}{4\pi}\qquad\quad\mbox{ if }
\mbox{Vol}_{g_0}(\R^2)<\infty\\
&\infty\qquad\qquad\qquad\mbox{ if }\mbox{Vol}_{g_0}(\R^2)=\infty
\end{aligned}\right.
\end{equation} 
and
\begin{equation}\label{vol}
\mbox{Vol}_{g(t)}(\R^2)=\left\{\begin{aligned}
&4\pi (T-t)\quad\forall 0\le t<T\quad\mbox{ if }\mbox{Vol}_{g_0}(\R^2)<\infty\\
&\infty\qquad\qquad\forall t>0\qquad\quad\mbox{ if }
\mbox{Vol}_{g_0}(\R^2)=\infty.
\end{aligned}\right.
\end{equation} 
If in addition the Gauss curvature satisfies \eqref{K-time0-upper-bd2}
for some constant $K_0\ge 0$, then the following holds.
\begin{enumerate}
\item[(i)] If $K_0>0$, then \eqref{K-upper-bd2} and 
\eqref{u-lower-bd1} holds on $\R^2$ for any $0\le t<(2K_0)^{-1}$ and
$T\ge(2K_0)^{-1}$.
\item[(ii)] if $K_0=0$, then \eqref{K-upper-bd3} and \eqref{u-lower-bd3} 
holds on $\R^2$ for any $t\ge 0$ and $T=\infty$.
\end{enumerate}
\end{thm}

We start with some definitions. For any $r_1>0$, $T_1>0$, let $B_{r_1}
=\{x\in\R^2:|x|<r_1\}$, $Q_{r_1}=B_{r_1}\times (0,\infty)$, 
 $Q_{r_1}^{T_1}=B_{r_1}\times (0,T_1)$, and $\1_pQ_{r_1}
=(\1 B_{r_1}\times [0,\infty))\cup (\2{B}_{r_1}\times\{0\})$. For any set 
$A\subset\R^2$, let $\chi_A$ be the characteristic function of the set $A$. 
Note that for any metric of the form $e^{2u(x,t)}\delta_{ij}$ in $Q_{r_1}^T$, 
we have 
$$
K[u]=-\frac{\Delta u}{e^{2u}}
$$
and $e^{2u(x,t)}\delta_{ij}$ is a solution of \eqref{Ricci-eqn} in $Q_{r_1}^T$ 
if any only if
\begin{equation}\label{u-eqn}
\frac{\1 u}{\1 t}=e^{-2u}\Delta u\quad\mbox{ in }Q_{r_1}^T.
\end{equation}
where $\Delta$ is the Euclidean Laplacian on $\R^2$. Let  
\begin{equation}\label{u-v-relation}
v=e^{2u}.
\end{equation}
Then \eqref{u-eqn} is equivalent to 
\begin{equation}\label{v-eqn}
\frac{\1 v}{\1 t}=\Delta\log v\quad\mbox{ in }Q_{r_1}^T.
\end{equation}
Existence and various properties of \eqref{v-eqn} were studied by 
P.~Daskalopoulos and M.A.~Del Pino \cite{DP1}, S.H.~Davis, E.~Dibenedetto, 
and D.J.~Diller \cite{DDD}, J.R.~Esteban, A.~Rodriguez, and J.L.~Vazquez
\cite{ERV1}, \cite{ERV2}, S.Y.~Hsu \cite{Hs1}, K.M.~Hui \cite{Hu2}, 
etc. We refer the 
readers to the book \cite{DK} by P.~Daskalopoulos and C.E.~Kenig
and the book \cite{V} by J.L.~Vazquez for the recent results
on the equation \eqref{v-eqn}. 

For any $0\le v_0\in L_{loc}^1(B_{r_1})$, we say that $v$ is a solution of 
\begin{equation}\label{v-boundary-infty-problem}
\left\{\begin{aligned}
&v_t=\Delta\log v\qquad\text{ in }Q_{r_1}\\
&v>0\qquad\qquad\,\,\,\mbox{ in }Q_{r_1}\\
&v(x,t)=\infty\qquad\mbox{ on }\1 B_{r_1}\times (0,\infty)\\
&v(x,0)=v_{0}(x)\quad\mbox{in }B_{r_1}\end{aligned}\right.
\end{equation}
if $v$ is a classical solution of \eqref{v-eqn} in $Q_{r_1}$,
\begin{equation*}\label{v-bdary}
\lim_{y\to x}v(y,t)=\infty\quad\forall x\in\1 B_{r_1}, t>0,
\end{equation*}
\begin{equation*}\label{v-inf}
\inf_{B_{r_1}\times (t_1,t_2)}v(x,t)>0\quad\forall t_2>t_1>0,
\end{equation*}
and
\begin{equation}\label{v-time0-value}
\lim_{t\to 0}\|v(\cdot,t)-v_0\|_{L^1(K)}=0
\end{equation}
for any compact set $K\subset B_{r_1}$. 
We say that $v$ is a solution of \eqref{v-eqn} in $\R^2\times (0,T)$
with initial value $v_0$ if $v$ is a classical solution of \eqref{v-eqn} 
in $\R^2\times (0,T)$, $v>0$ in $\R^2\times (0,T)$, and
\eqref{v-time0-value} holds for any compact set $K\subset\R^2$. 

Note that for any solution $v$ of \eqref{v-eqn} the equation \eqref{v-eqn} 
is uniformly parabolic on $K$ for any compact set $K\subset Q_{r_1}^T$. 
Hence by the Schauder estimates \cite{LSU} and a bootrap argument 
$v\in C^{\infty}(Q_{r_1}^T)$ for any solution $v$ of
\eqref{v-eqn}. 
We first observe that by the same argument as the proof of Theorem 1.6
of \cite{Hu3} we have the following result.

\begin{thm}\label{hui-result1}(cf. Theorem 1.6 of \cite{Hu3}) 
Let $r_1>0$. Suppose $v_0\ge 0$ satisfies
\begin{equation*}\label{v-initial-cond}
v_0\in L_{loc}^p(B_{r_1})\quad\mbox{ for some constants }p>1. 
\end{equation*} 
Then exists a solution $v$ of \eqref{v-boundary-infty-problem}
which satisfies
\begin{equation}\label{aronson-benilan}
v_t\le\frac{v}{t}
\end{equation}
in $Q_{r_1}$ and
\begin{equation*}
v(x,t)\ge C\frac{t}{\phi (x)}\quad\mbox{ in }Q_{r_1}
\end{equation*}
for some constant $C>0$ depending on $\phi$ and $\lambda$ where 
$\phi$ and $\lambda$ are the first positive eigenfunction 
and the first positive eigenvalue of the Laplace operator 
$-\Delta$ on $B_{r_1}$ with $\|\phi\|_{L^2(B_{r_1})}=1$. 
\end{thm}

\begin{lem}\label{v-smooth}
Let $r_1>0$. Suppose $v_0$ is a positive smooth function on $B_{r_1}$ 
and $v$ is a solution of \eqref{v-boundary-infty-problem}.
Then $v\in C^{\infty}(B_{r_1}\times [0,\infty))$. 
\end{lem}
\begin{proof}
We will first use a modification of the technique of 
Lemma 1.6 of \cite{Hu1} to show that $v$ is uniformly bounded below 
by a positive constant on $\2{B}_{r_3}\times [0,1]$ for any $0<r_3<r_1$. 
Let $0<r_3<r_2<r_1$, $\delta_1=r_3-r_2$, and 
\begin{equation*}
\3=\frac{1}{2}\min_{\2{B}_{r_2}}v_0.
\end{equation*}
Since $v_0$ is a smooth positive function on $B_{r_1}$, $\3>0$. Let $w$ be 
the maximal solution of the equation \eqref{v-eqn} in $\R^2\times (0,T)$
constructed in \cite{DP1} and \cite{Hu2} with initial value $w(x,0)
=\3\chi_{\2{B}_{\delta_1}}$ and $T=\3|\2{B}_{\delta_1}|/4\pi$ which satisfies
\begin{equation*}
\int_{\R^2}w(x,t)\,dx=\int_{\R^2}w(x,0)\,dx-4\pi t\quad\forall 0\le t<T.
\end{equation*}
For any $\alpha>0$, let
\begin{equation*}
w_{\alpha}(x,t)=w(\alpha x,\alpha^2t)\quad\forall x\in\R^2,
0\le t<T/\alpha^2.
\end{equation*}
and
\begin{equation*}
v_{\alpha,p}(x,t)=v(\alpha x+p,\alpha^2t)\quad\forall p\in\2{B}_{r_3},
|x|<\delta_1/\alpha,t>0.
\end{equation*}
Since $v_0(x+p)\ge w(x,0)$ for any $p\in\2{B}_{r_3}$, $x+p\in B_{r_1}$, 
and $v=\infty$ on $\1 B_{r_1}\times (0,\infty)$, 
by Corollary 1.8 and Lemma 2.9 of \cite{Hu4},
\begin{align}\label{v-alpha-w-alpha-compare}
&v(x+p,t)\ge w(x,t)\quad\forall p\in\2{B}_{r_3},x+p\in B_{r_1},0\le t<T
\nonumber\\
\Rightarrow\quad&v_{\alpha,p}(x,t)\ge w_{\alpha}(x,t)\quad\forall 
p\in\2{B}_{r_3},|x|<\delta_1/\alpha, 0\le t<T/\alpha^2.
\end{align}
Since $w_{\alpha}(x,0)=w(\alpha x,0)=\3\chi_{\2{B}_{\delta_1/\alpha}}$,
\begin{equation}\label{w-alpha-time-0-compare}
0\le w_{\alpha_1}(x,0)\le w_{\alpha_2}(x,0)\le\3\quad\forall x\in\R^2,
\alpha_1\ge\alpha_2>0.
\end{equation}
By \eqref{w-alpha-time-0-compare} and the comparison theorems in 
\cite{DP1} and \cite{Hu2},
\begin{equation}\label{w-alpha-compare}
0<w_{\alpha_1}(x,t)\le w_{\alpha_2}(x,t)\le\3\quad\forall x\in\R^2,
0<t<\3|\2{B}_{\delta_1}|/(4\pi\alpha_1^2),\alpha_1\ge\alpha_2>0.
\end{equation}
We choose $0<\alpha_0<\delta_1$ such that 
$\3|\2{B}_{\delta_1}|/(4\pi\alpha_0^2)>1$.
By \eqref{w-alpha-compare} the equation \eqref{v-eqn} for 
$w_{\alpha}$, $0<\alpha\le\alpha_0$, is uniformly parabolic on every
compact set $K\subset\R^2\times (0,1]$. Then by the Schauder estimates 
\cite{LSU} $\{w_{\alpha}\}_{\{0<\alpha\le\alpha_0\}}$ are equi-Holder
continuous in $C^2(K)$ for any compact set $K\subset\R^2\times (0,1]$.
Hence by \eqref{w-alpha-compare} $w_{\alpha}$ increases and converges 
uniformly on every compact subset $K$ of $\R^2\times (0,1]$ to the 
constant function $\3$ in $\R^2\times (0,1]$ as $\alpha\to 0$. Then
there exists $\alpha_1\in (0,\alpha_0)$ such that
\begin{equation}\label{w-alpha-lower-bd}
w_{\alpha}(x,1)\ge\frac{\3}{2}\quad\forall |x|\le 1,0<\alpha\le\alpha_1.
\end{equation}
By \eqref{v-alpha-w-alpha-compare} and \eqref{w-alpha-lower-bd},
\begin{align}\label{v-compact-set-lower-bd}
&v(\alpha x+p,\alpha^2)=v_{\alpha,p}(x,1)\ge\frac{\3}{2}
\quad\forall p\in\2{B}_{r_3},|x|\le 1,0<\alpha\le\alpha_1\nonumber\\
\Rightarrow\quad&v(p+y,t)\ge\frac{\3}{2}\quad\forall p\in\2{B}_{r_3},
|y|\le\sqrt{t},0<t\le\alpha_1^2.
\end{align}
Let $\lambda$ and $\phi$ be the first positive eigenvalue and the 
first positive eigenfunction of $-\Delta$ on $B_{r_2}$ with 
$\|\phi\|_{L^2(B_{r_2})}=1$. Let 
$C_1=\max (2\lambda\|\phi\|_{L^{\infty}},12\|\nabla\phi\|_{L^{\infty}},
\|v_0\|_{L^{\infty}(\2{B}_{r_1})}+1)$ and
\begin{equation*}
\psi(x,t)=C_1(t+1)e^{1/\phi(x)}.
\end{equation*}
By the computation on P.784-785 of \cite{Hu3}, $\psi$ is a supersolution
of \eqref{v-eqn}. Then an argument similar to the proof Theorem 2.9 of 
\cite{Hu4},
\begin{equation}\label{v-psi-compare}
\int_{B_{r_2}}(v-\psi)_+(x,t)\,dx\le\int_{B_{r_2}}(v-\psi)_+(x,t_1)\,dx
\quad\forall t\ge t_1>0.
\end{equation}
Since
\begin{align*}
\int_{B_{r_2}}(v-\psi)_+(x,t_1)\,dx\le&\int_{B_{r_2}}|v(x,t_1)-v_0(x)|\,dx
+\int_{B_{r_2}}(v_0(x)-\psi(x,t_1))_+\,dx\nonumber\\
\to&0\qquad\mbox{ as }t_1\to 0,
\end{align*}
letting $t_1\to 0$ in \eqref{v-psi-compare},
\begin{equation*}
\int_{B_{r_2}}(v-\psi)_+(x,t)\,dx=0\quad\forall t>0.
\end{equation*}
Hence
\begin{equation*}
v(x,t)\le\psi(x,t)=C_1(t+1)e^{1/\phi(x)}\quad\forall |x|<r_2,t\ge 0.
\end{equation*}
Thus
\begin{equation}\label{v-upper-bd2}
v(x,t)\le C'(t+1)\quad\forall |x|\le r_3,t\ge 0
\end{equation}
for some constant $C'>0$.
By \eqref{v-compact-set-lower-bd} and \eqref{v-upper-bd2}, the equation
\eqref{v-eqn} is uniformly parabolic on $\2{B}_{r_3}\times [0,1]$ for
any $0<r_3<r_1$. By standard parabolic estimates \cite{LSU} and a 
bootstrap argument, $v\in C^{\infty}(B_{r_1}\times [0,\infty))$ and the
lemma follows. 
\end{proof}

We will now prove Theorem \ref{thm-a1}.

\ni{\it Proof of Theorem \ref{thm-a1}}:
Let $v_0(x)=e^{2u_0(x)}$.
For any $k=2,3,\dots$, by Theorem \ref{hui-result1} there exists a 
solution $v_k$ of \eqref{v-boundary-infty-problem} with $r_1=1-(1/k)$
and initial value $v_0$ which satisfies \eqref{aronson-benilan} in 
$B_{1-(1/k)}\times (0,\infty)$ and  
\begin{equation}\label{vk-lower-bd}
v_k(x,t)\ge C_k\frac{t}{\phi_k(x)}\quad\mbox{ in }B_{1-(1/k)}\times (0,\infty)
\end{equation}
for some constant $C_k>0$ depending on $\phi_k$ and $\lambda_k$ where 
$\phi_k$ and $\lambda_k$ are the first positive eigenfunction 
and the first positive eigenvalue of the Laplace operator 
$-\Delta$ on $B_{1-(1/k)}$ with $\|\phi\|_{L^2(B_{1-(1/k)})}=1$.
By Lemma \ref{v-smooth}, $v_k\in C^{\infty}(B_{1-(1/k)}\times [0,\infty))$.  
By Corollary 1.8 and Theorem 2.9 of \cite{Hu4},
\begin{equation}\label{vk-v(k+1)-compare}
v_k\ge v_{k+1}\quad\mbox{ in }B_{1-(1/k)}\times (0,\infty)
\quad\forall k\ge 2.
\end{equation} 

Since $\phi_k$ and $\lambda_k$ converges to $\phi$ and $\lambda$ as 
$k\to\infty$ where $\phi$ and $\lambda$ are the first positive eigenfunction 
and the first positive eigenvalue of the Laplace operator $-\Delta$ 
on $B_1$ with $\|\phi\|_{L^2(B_1)}=1$,  by the proof of Theorem 1.2 and 
Theorem 1.6 of \cite{Hu3} there exists a constant 
$C>0$ such that
\begin{equation}\label{ck-bd}
C_k\ge C\quad\forall k\ge 2.
\end{equation}
Let $k_0\ge 2$. Then by \eqref{vk-lower-bd} and \eqref{ck-bd} there exists 
a constant $C_0>0$ such that
\begin{equation}\label{vk-lower-bd2}
v_k(x,t)\ge C_0t\quad\mbox{ in }B_{1-(1/k_0)}\times (0,\infty)\quad\forall
k>k_0.
\end{equation}
By Corollary 1.8 of \cite{Hu4} for any $0<t_1<t_2$ there exists a 
constant $C'>0$ such that 
\begin{equation}\label{vk-upper-bd}
v_k(x,t)\le C'\quad\forall |x|\le 1-(1/k_0),t_1\le t\le t_2,k>k_0.
\end{equation}
By \eqref{vk-lower-bd2} and \eqref{vk-upper-bd} the equation 
\eqref{v-eqn} for $v_k$ are uniformly parabolic on $\2{Q}_{1-(1/k_0)}$ for all 
$k>k_0$. By the parabolic Schauder estimates \cite{LSU} $v_k$ are 
equi-Holder continuous in 
$C^2(K)$ for every compact set $K\subset\2{Q}_{1-(1/k_0)}$ 
and all $k>k_0$. Hence by \eqref{vk-v(k+1)-compare} $v_k$ decreases 
and converges uniformly on $K$ to a solution $v$ of \eqref{v-eqn} in 
$\mcD\times (0,\infty)$ as $k\to\infty$ for any compact set 
$K\subset\mcD\times (0,\infty)$. Since $v_k$ satisfies 
\eqref{aronson-benilan} in $Q_{1-(1/k)}$ for any $k\in\Z^+$, 
by an argument similar to the proof of Theorem 2.4 of \cite{Hu3} $v$ 
has initial value $v_0$ and satisfies \eqref{aronson-benilan} in 
$\mcD\times (0,\infty)$. Letting $k\to\infty$ in \eqref{vk-lower-bd}, 
by \eqref{ck-bd},
\begin{align*}
&v(x,t)\ge C\frac{t}{\phi(x)}\quad\mbox{ in }\mcD\times (0,\infty)\\
\Rightarrow\quad&\lim_{x\to x_0}v(x,t)=\infty\quad\forall |x_0|=1,t>0.
\end{align*}
Hence $v$ is a solution of \eqref{v-boundary-infty-problem} with $B_{r_1}
=\mcD$. By Lemma \ref{v-smooth}, $v\in C^{\infty}(\mcD\times [0,\infty))$.

Let $u$ be given by \eqref{u-v-relation} and $g(t)= e^{2u}\delta_{ij}$. 
Then $g(t)$ is a smooth solution of \eqref{Ricci-eqn} in $\mcD\times 
[0,\infty)$ with initial value
$e^{2u_0}\delta_{ij}$. By \eqref{v-eqn} and \eqref{aronson-benilan}, we get 
\eqref{K-lower-bd}. By \eqref{K-lower-bd},
\begin{equation}\label{u-sup-elliptic}
\Delta u(x,t)\le\frac{1}{2t}e^{2u(x,t)}\quad\mbox{ in }\mcD\quad\forall
t>0.
\end{equation}
For any $0<\delta<1$, let 
\begin{equation*}\label{psi-defn}
\psi_{\delta}(x,t)=\log\frac{2(1+\delta)}{(1+\delta)^2-|x|^2}
+\frac{1}{2}\log (2t)\quad\forall x\in\mcD,t>0.
\end{equation*}
Then $\psi_{\delta}$ satisfies
\begin{equation}\label{psi-eqn}
\Delta\psi_{\delta}(x,t)=\frac{1}{2t}e^{2\psi_{\delta}(x,t)}\quad\mbox{ in }
\mcD\quad\forall t>0.
\end{equation}
Let $\Omega(t)=\{x\in\mcD:u(x,t)<\psi_{\delta}(x,t)\}$.
By \eqref{u-sup-elliptic} and \eqref{psi-eqn},
\begin{equation}\label{u-psi-compare-eqn}
\Delta(u(x,t)-\psi_{\delta}(x,t))\le\frac{1}{2t}(e^{2u(x,t)}-e^{2\psi_{\delta}(x,t)})
<0\quad\mbox{ in }\Omega(t)\quad\forall t>0.
\end{equation}
Since $u(x,t)-\psi_{\delta}(x,t)\to\infty$ as $|x|\to 1$, $\2{\Omega(t)}
\subset\mcD$. Hence by \eqref{u-psi-compare-eqn} and the maximum principle 
\cite{GT},
\begin{align}\label{u-psi-compare-eqn2}
&u(x,t)-\psi_{\delta}(x,t)\ge\min_{x\in\1 \Omega(t)}(u(x,t)-\psi_{\delta}(x,t))
=0\quad\mbox{ in }\Omega(t)\quad\forall t>0\nonumber\\
\Rightarrow\quad&u(x,t)\ge\psi_{\delta}(x,t)\quad\mbox{ in }\mcD\quad\forall
t>0.
\end{align}
Letting $\delta\to 0$ in \eqref{u-psi-compare-eqn2}, \eqref{u-lower-bd}
follows. We now let $v_{k,m}$ be the solution of  
\begin{equation}\label{vk-boundary-m-problem}
\left\{\begin{aligned}
&v_t=\Delta\log v\qquad\qquad\quad\text{in }Q_{1-(1/k)}\\
&v>0\qquad\qquad\qquad\quad\,\,\,\mbox{ in }Q_{1-(1/k)}\\
&v(x,t)=v_0(x)e^{mt^2+2t}\quad\mbox{on }\1 B_{1-(1/k)}\times (0,\infty)\\
&v(x,0)=v_0(x)\qquad\quad\,\,\mbox{ in }B_{1-(1/k)}
\end{aligned}\right.
\end{equation}
for any $k\ge 2$, $m\in\Z^+$, which can be constructed by similar technique as
that of \cite{Hu3}. By an argument similar to the proof of \cite{Hu3} 
$v_{k,m}$ increases and converges uniformly to $v_k$ on every compact subset
$K$ of $Q_{1-(1/k)}$ as $m\to\infty$ and
$v_{k,m}$ satisfies
\begin{equation}\label{vkm-lower-upper-bd}
0<\min_{|x|\le 1-k^{-1}}v_0(x)\le v_{k,m}(x,t)\le e^{mT_1^2+2T_1}\max_{|x|\le 1-k^{-1}}v_0(x)
\quad\forall |x|\le 1-k^{-1},0\le t\le T_1,k\ge 2,
\end{equation}  
for any $T_1>0$. By \eqref{vkm-lower-upper-bd} the equation \eqref{v-eqn}
for $v_{k,m}$ is uniformly parabolic on $\2{Q_{1-(1/k)}^{T_1}}$ for any $T_1>0$. 
Hence by \eqref{vk-boundary-m-problem}, \eqref{vkm-lower-upper-bd}, 
and parabolic estimates \cite{LSU}, 
$v_{k,m}\in C^{\infty}(\2{Q}_{1-(1/k)})$. Let $p_{k,m}=\1_tv_{k,m}/v_{k,m}$. Then 
$p_{k,m}\in C^{\infty}(\2{Q}_{1-(1/k)})$. By \eqref{K-time0-upper-bd} and 
\eqref{vk-boundary-m-problem} $p_{k,m}$ satisfies
\begin{equation}\label{pkm-eqn}
p_t=e^{-v_{k,m}}\Delta p-p^2\quad\text{ in }Q_{1-(1/k)}
\end{equation}
and
\begin{equation}\label{pkm-bdary}
\left\{\begin{aligned}
&p=2mt+2\quad\,\mbox{ on }\1 B_{1-(1/k)}\times [0,\infty)\\
&p(x,0)\ge 2\qquad\forall |x|\le 1-k^{-1}. 
\end{aligned}\right.
\end{equation}
Note that the function $2/(2t+1)$ satisfies \eqref{pkm-eqn}
and by \eqref{pkm-bdary},
\begin{equation*}
p_{k,m}\ge 2/(2t+1)\quad\text{ on }\1_pQ_{1-(1/k)}.
\end{equation*}
Hence by the maximum principle (cf. \cite{A}),
\begin{align*}\label{vt-v-upper-bd}
&\frac{\1_tv_{k,m}}{v_{k,m}}=p_{k,m}(x,t)\ge\frac{2}{2t+1}\quad\text{ in }
Q_{1-(1/k)}\quad\forall k\ge 2,m\ge 1\\
\Rightarrow\quad&\frac{v_t}{v}\ge\frac{2}{2t+1}\qquad\qquad\qquad
\qquad\text{ in }\mcD\times [0,\infty)\quad\mbox{ as }m\to\infty, 
k\to\infty,
\end{align*}
and \eqref{K-upper-bd} follows. By \eqref{K-upper-bd}
and an argument similar to the proof of \eqref{u-lower-bd} we get 
\eqref{u-upper-bd}. By \eqref{u-lower-bd} $g(t)$ is complete for any $t>0$.
Now by \eqref{K-upper-bd},
\begin{align*}
&K[u(t)]\le -\frac{1}{2t+1}\qquad\qquad\quad\mbox{ in }
\mcD\times [0,\infty)\\
\Rightarrow\quad&\frac{\1 u}{\1 t}\ge\frac{1}{2t+1}\qquad\qquad
\qquad\qquad\mbox{ in }\mcD\times [0,\infty)\\
\Rightarrow\quad&u(x,t)\ge u_0(x)+\frac{1}{2}\log(2t+1)
\quad\mbox{ in }\mcD\times [0,\infty) 
\end{align*}
and \eqref{u-u0-ineqn} follows. Suppose $\4{g}(t)$, $0\le t\le\3$, is a 
solution of \eqref{Ricci-eqn} in $\mcD\times (0,\3)$ with $g(0)=g_0$. 
As in \cite{GiT1} we can write $\4{g}=e^{2\4{u}}\delta_{ij}$. Let 
$\4{v}=e^{2\4{u}}$. Then $\4{v}$ is a solution of 
\eqref{v-boundary-infty-problem} with $r_1=1$. Hence by Corollary 1.8
and Theorem 2.9 of \cite{Hu4},
\begin{align*}
&\4{v}(x,t)\le v_k(x,t)\quad\forall |x|\le 1-(1/k),0\le t\le\3\\
\Rightarrow\quad&\4{v}(x,t)\le v(x,t)\quad\forall |x|<1,0\le t\le\3
\quad\mbox{ as }
k\to\infty
\end{align*}
and \eqref{tilde-g-and-g-compare} follows. 

{\hfill$\square$\vspace{6pt}}

\ni{\it Proof of Theorem \ref{thm-a4}}:
Let $v_0=e^{2u_0},v_k,v_{k,m},v,u, g(t),p_{k,m}$ be as in the proof of 
Theorem \ref{thm-a2}. By the proof of Theorem \ref{thm-a2} 
$v\in C^{\infty}(\mcD\times [0,\infty))$ satisfies \eqref{K-lower-bd} and
\eqref{u-lower-bd} and $g(t)$ is a smooth maximal instanteous complete 
solution of \eqref{Ricci-eqn} for all time $t\ge 0$. 

Suppose now \eqref{K-time0-upper-bd2} holds for some constant $K_0\ge 0$. 
We will show that the curvature $K[u(t)]$ satisfies \eqref{K-upper-bd2}. 
Let $T_0=(2K_0)^{-1}$. By \eqref{K-time0-upper-bd2} and 
\eqref{vk-boundary-m-problem}, $p_{k,m}$ satisfies
\begin{equation}\label{pkm-lower-bd-p-bdy}
\left\{\begin{aligned}
&p(x,t)=2mt+2\ge -\frac{1}{(2K_0)^{-1}-t}\quad\forall |x|=1-k^{-1},t\ge 0\\
&p(x,0)\ge -2K_0\qquad\qquad\qquad\qquad\quad\forall |x|\le 1-k^{-1}.
\end{aligned}\right.
\end{equation}
Since both $p_{k,m}$ and $-1/((2K_0)^{-1}-t)$ satisfy \eqref{pkm-eqn}, 
by \eqref{pkm-lower-bd-p-bdy} and the maximum principle,
\begin{align*}
&\frac{\1_tv_{k,m}}{v_{k,m}}=p_{k,m}(x,t)
\ge-\frac{1}{(2K_0)^{-1}-t}\quad\forall |x|<1-(1/k),0\le t<(2K_0)^{-1}\\
\Rightarrow\quad&\frac{v_t}{v}\ge-\frac{1}{(2K_0)^{-1}-t}
\quad\forall |x|<1-(1/k),0\le t<(2K_0)^{-1}\quad\mbox{ as }m\to\infty,
k\to\infty
\end{align*}
and \eqref{K-upper-bd2} follows. By \eqref{K-upper-bd2} and 
\eqref{u-eqn},
\begin{equation*}
u_t(x,t)\ge -\frac{1}{K_0^{-1}-2t}\quad\forall |x|<1,0<t<(2K_0)^{-1}. 
\end{equation*}
Integrating the above equation with respect to t and \eqref{u-lower-bd1} 
follows.

Suppose now \eqref{K-time0-upper-bd2} holds with $K_0=0$. Then by 
\eqref{K-time0-upper-bd2} and \eqref{vk-boundary-m-problem},
\begin{equation*}
p_{k,m}\ge 0\quad\mbox{ on }\1_pQ_{1-(1/k)}.
\end{equation*} 
Hence by the maximum principle,
\begin{align*}
&\frac{\1_tv_{k,m}}{v_{k,m}}=p_{k,m}(x,t)
\ge 0\quad\forall |x|<1-k^{-1},t\ge 0\\
\Rightarrow\quad&v_t\ge 0\quad\forall |x|<1,t\ge 0
\quad\mbox{ as }m\to\infty,k\to\infty
\end{align*}
and \eqref{K-upper-bd3} follows. By \eqref{K-upper-bd3} and 
\eqref{u-v-relation},
\begin{equation*}
u_t(x,t)\ge 0\quad\forall |x|<1,t>0
\end{equation*}
and \eqref{u-lower-bd3} follows.

{\hfill$\square$\vspace{6pt}}

\ni{\it Proof of Theorem \ref{thm-a5}}:
We will use a modification of the technique of \cite{DP2} and \cite{Hu3}
to prove the theorem. Let $v_0=e^{2u_0}$. by the results of \cite{DP1},
\cite{Hu2}, and \cite{Hs1}, there exists a unique maximal solution $v$ 
of \eqref{v-eqn} in $\R^2\times (0,T)$ which satisfies \eqref{vol} and
\eqref{aronson-benilan} in $\R^2\times (0,T)$ where $T$ is given by
\eqref{T-value}. By Theorem 3.4 of \cite{ERV2}
and the result of \cite{Hs1} for any $0<T_1<T$ and $r_0>1$ there exists a 
constant $C>0$ such that 
\begin{equation}\label{v-special-soln-ineqn}
v(x,t)\ge\frac{Ct}{|x|^2(\log |x|)^2}\quad\forall |x|\ge r_0,0\le t<T.
\end{equation}
Let $u$ be given by \eqref{u-v-relation} and $g(t)=e^{2u}\delta_{ij}$. 
Then by \eqref{aronson-benilan} and \eqref{v-special-soln-ineqn}, 
\eqref{K-lower-bd} and \eqref{u-lower-bd2} hold.

For any $k\in\Z^+$, by the same argument as the proof of Theorem 
\ref{thm-a4} there exists a solution $\4{v}_k$ of 
\eqref{v-boundary-infty-problem} with $r_1=k$ which satisfies 
\eqref{aronson-benilan} in $B_k\times (0,\infty)$. 
By Corollary 1.8 and Theorem 2.9 of \cite{Hu4},
\begin{equation}\label{tilde-vk-v(k+1)-compare}
\4{v}_k\ge\4{v}_{k+1}\ge v\quad\mbox{ in }B_k\times (0,T)
\quad\forall k\in\Z^+.
\end{equation} 
Let $k_0\in\Z^+$ and $0<t_1<t_2<T$. By Corollary 1.8 of \cite{Hu4} there 
exists a constant $C>0$ such that 
\begin{equation}\label{tilde-vk-upper-bd}
\4{v}_k\le C\quad\mbox{ in }\2{B}_{k_0}\times [t_1,t_2]\quad\forall k>k_0.
\end{equation} 
Hence by \eqref{tilde-vk-v(k+1)-compare} and \eqref{tilde-vk-upper-bd}
the equation \eqref{v-eqn} for $\4{v}_k$ is uniformly parabolic on 
$\2{B}_{k_0}\times [t_1,t_2]$ for any $k>k_0$. By the Schauder estimates
\cite{LSU} the sequence $\{\4{v}_k\}_{\{k>k_0\}}$ is equi-Holder 
continuous in 
$C^2(\2{B}_{k_0}\times [t_1,t_2])$. Hence by \eqref{tilde-vk-v(k+1)-compare} as 
$k\to\infty$, $\4{v}_k$ decreases and converges unformly on every compact 
subset $K$ of $\R^2\times (0,T)$ to a solution $\4{v}$ of \eqref{v-eqn}
in $\R^2\times (0,T)$.
Similar to the proof of Theorem 1.2 of \cite{Hu3} $\4{v}$ has initial
value $v_0$. By \eqref{tilde-vk-v(k+1)-compare},
\begin{equation}\label{v-tilde-v-compare}
\4{v}\ge v\quad\mbox{ in }\R^2\times (0,T)\quad\forall k\in\Z^+.
\end{equation} 
Since $v$ is the maximal solution of \eqref{v-eqn} in $\R^2\times (0,T)$
with initial value $v_0$, by \eqref{v-tilde-v-compare}, 
$\4{v}\equiv v$ in $\R^2\times [0,T)$. 

Since $v_0>0$ on $\R^2$, by  \eqref{tilde-vk-v(k+1)-compare} and 
an argument similar to the proof of Lemma \ref{v-smooth},
\begin{equation*}
0<\inf_{\2{B}_{r_1}\times [0,1]}v\le\sup_{\2{B}_{r_1}\times [0,1]}v<\infty
\quad\forall r_1>0
\end{equation*} 
and $v\in C^{\infty}(\R^2\times [0,T))$. Then $g(t)$ is the smooth 
maximal solution of \eqref{Ricci-eqn} in $0\le t<T$ with initial value 
$g_0$. By \eqref{u-lower-bd2}, $g(t)$ is complete for any $0<t<T$.

Finally if in addition the Gauss curvature satisfies \eqref{K-time0-upper-bd2}
for some constant $K_0\ge 0$, then by an argument similar to the proof
of Theorem \eqref{thm-a4} we get that \eqref{K-upper-bd2} and
\eqref{u-lower-bd1} hold on $\R^2$ for any $0\le t<\min ((2K_0)^{-1},T)$ if 
$K_0>0$ and \eqref{K-upper-bd3} and \eqref{u-lower-bd3} hold on $\R^2$ for any 
$0\le t<T$ if $K_0=0$. By \eqref{u-lower-bd3}, $T=\infty$ if $K_0=0$.

If $K_0>0$, we claim that $T\ge (2K_0)^{-1}$. Suppose not. Then $T<(2K_0)^{-1}$.
Hence by \eqref{u-lower-bd1},
\begin{equation}
\mbox{Vol}_{g(T)}(\R^2)=\int_{\R^2}v(x,T)\,dx\ge(1-2K_0T)\int_{\R^2}v_0(x)\,dx>0.
\end{equation}
This contradicts \eqref{vol}. Hence $T\ge (2K_0)^{-1}$ and the theorem follows.

{\hfill$\square$\vspace{6pt}}

\end{document}